\documentclass[a4, reqno, 12pt]{amsart}
\usepackage{amssymb}
\usepackage{amstext}
\usepackage{amsmath}
\usepackage{amscd}
\usepackage{amsthm}
\usepackage{amsfonts}
\usepackage{enumerate}
\usepackage{graphicx}
\usepackage{latexsym}
\usepackage[all]{xy}
\usepackage{color}

\theoremstyle{plain}
\newtheorem{thm}{Theorem}[section]
\newtheorem*{thm*}{Theorem}
\newtheorem*{cor*}{Corollary}
\newtheorem*{prop*}{Proposition}

\newtheorem{prop}[thm]{Proposition}
\newtheorem{lemma}[thm]{Lemma}
\newtheorem{cor}[thm]{Corollary}

\newtheorem*{claim*}{Claim}

\theoremstyle{definition}
\newtheorem{definition}[thm]{Definition}
\newtheorem{remark}[thm]{Remark}
\newtheorem{ex}[thm]{Example}

\newtheorem*{conj*}{Conjecture}

\newtheorem*{ac}{{\sc Acknowledgments}}
\newtheorem*{mpf}{Proof of Theorem \ref{thm:mutation of Brauer tree}}
\newtheorem{method}[thm]{Method}

\newtheorem*{ques*}{Question}

\numberwithin{equation}{thm}

\begin{document}

\setlength{\baselineskip}{15pt}


\newcommand{\modu}[1]{{\rm{mod}}\text{-}#1} 
\newcommand{\proj}[1]{{\rm{proj}}\text{-}#1} 
\newcommand{\inj}[1]{{\rm{inj}}\text{-}#1} 
\newcommand{\add}[1]{\mathsf{add}#1} 

\newcommand{\sta}[1]{{\underline{{\rm{mod}}}}\text{-}#1} 
\newcommand{\ista}[1]{{\overline{{\rm{mod}}}}\text{-}#1} 

\newcommand{\Modu}[1]{{\rm{Mod}}\text{-}#1} 
\newcommand{\Proj}[1]{{\rm{Proj}}\text{-}#1} 
\newcommand{\Inj}[1]{{\rm{Inj}}\text{-}#1} 

\newcommand{\homo}[3]{{\rm{Hom}}_{#1}(#2, #3)} 
\newcommand{\ext}[4]{{\rm{Ext}}_{#1}^{#2}(#3, #4)} 
\newcommand{\tor}[4]{{\rm{Tor}}^{#1}_{#2}(#3, #4)} 
\newcommand{\ehomo}[2]{{\rm{End}}_{#1}(#2)} 
\newcommand{\shomo}[3]{\underline{{\rm{Hom}}}_{#1}(#2, #3)} 
\newcommand{\ishomo}[3]{{\overline{{\rm{Hom}}}}_{#1}(#2, #3)} 
\newcommand{\thomo}[3]{{\rm{Hom}}^{#1}(#2, #3)} 
\newcommand{\rhomo}[3]{{\mathbf{R}}{\rm{Hom}}^{#1}(#2, #3)} 
\newcommand{\syz}[2]{\Omega ^{#1}(#2)} 
\newcommand{\thick}[1]{\mathsf{thick}#1} 
\newcommand{\KK}{\mathsf{K}} 
\renewcommand{\dim}[1]{{\rm{dim}}_{k}#1} 
\title{Mutating Brauer trees}
\author{Takuma Aihara}
\address{Division of Mathematical Science and Physics, Graduate School of Science and Technology, Chiba University, 
Yayoi-cho, Chiba 263-8522, Japan}
\email{taihara@math.s.chiba-u.ac.jp}
\begin{abstract}
In this paper we introduce mutation of Brauer trees. 
We show that our mutation of Brauer trees explicitly describes the tilting mutation of Brauer tree algebras 
introduced by Okuyama and Rickard. 
\end{abstract}
\maketitle

Mutation plays an important role in representation theory, especially in tilting theory and cluster tilting theory. 
Cluster tilting theory deals with the combinatorial structure of 2-Calabi-Yau triangulated categories 
and is applied to categorification of Fomin-Zelevinsky cluster algebras.
This is closely related with tilting theory of hereditary algebras, 
and the class of tilting complexes is one of the most important classes from Morita theoretic viewpoint \cite{Ri2}. 
Now it is an important problem to study the combinatorial structure of tilting complexes for finite dimensional algebras. 
In this paper we consider this problem for Brauer tree algebras, 
which form one of the most basic classes of symmetric algebras. 
The main result of this paper is to describe explicitly the combinatorics of tilting mutation for Brauer tree algebras. 
This is given by mutation of Brauer trees, which is a new operation introduced in this paper. 

In Section \ref{pre} we recall a construction of tilting complexes $T$
of symmetric algebras $A$ which we call \emph{Okuyama-Rickard complexes}.
It was introduced by Okuyama and Rickard, and has been played an important role 
in the study of Brou\'{e}'s abelian defect group conjecture. 
Since it is a special case of tilting mutation as
we pointed out in [AI], we call the endomorphism algebra $\ehomo{\KK^{\rm b}(\proj{A})}{T}$
\emph{tilting mutation} of $A$.

In Section \ref{main} we introduce a combinatorial operation of Brauer trees which we call \emph{mutation of Brauer trees}. 
Our main result shows that tilting mutation of Brauer tree algebras which we explained above
is compatible with our mutation of Brauer trees. 
A special case was given in \cite{KZ} and \cite{Z} (see Remark \ref{KZ and Z}). 
As an application of our main result, we give braid-type relations for tilting mutation of Brauer tree algebras 
(Theorem \ref{tilting formula}). 

In Section \ref{proof} we prove our main result. 
Our proof is very simple and  seems to be interesting by itself.

\section{Preliminary}\label{pre}

Through this paper, let $A$ be a finite dimensional $k$-algebra for an algebraically closed field $k$ 
and we assume that $A$ is basic and indecomposable as an $A$-$A$-bimodule. 
Let $\{e_{1}, e_{2}, \cdots, e_{n}\}$ be a basic set of orthogonal local idempotents in $A$ and 
put $E=\{1, 2, \cdots, n\}$. 
For each $i\in E$, we set $P_{i}=e_{i}A$ and $S_{i}=P_{i}/{\rm{rad}}P_{i}$. 

We denote by $\modu{A}$ the category of finitely generated right $A$-modules, 
by $\proj{A}$ the full subcategory of $\modu{A}$ consisting of finitely generated projective right $A$-modules, 
by $\sta{A}$ the stable module category of $\modu{A}$ and 
by $\KK^{\rm{b}}(\proj{A})$ the homotopy category of bounded complexes over $\proj{A}$. 

Let us start with recalling the complex introduced by Okuyama and Rickard \cite{Ok, Ri}, 
which is a special case of silting mutation defined in \cite{AI}. 

\begin{definition}
Let $E_{0}$ be a subset of $E$ and put $e=\sum_{i \in E_{0}} e_i$. 
For any $i\in E$, we define a complex by 
\[T_{i}=\begin{cases}
\begin{array}{c}
\xymatrix@!R=0.5pt{
 (0{\rm{th}}) & (1{\rm{st}}) & \\
 P_{i} \ar[r] & 0 & (i\in E_{0})  \\
 Q_{i} \ar[r]_{\pi_{i}} & P_{i} & (i\not\in E_{0})
}
\end{array}
\end{cases}
\]
where $Q_{i}\xrightarrow{\pi_{i}}P_{i}$ is a minimal projective presentation of $e_{i}A/e_{i}AeA$. 
Now we define $T:=T(E_0):=\oplus_{i\in E} T_{i}$ and 
call it the \emph{Okuyama-Rickard complex with respect to $E_{0}$}. 
\end{definition}

The following observation shows the importance of Okuyama-Rickard complexes. 

\begin{prop}\label{prop:ok}\cite[Proposition 1.1]{Ok} 
If $A$ is a symmetric algebra, 
then any Okuyama-Rickard complex $T$ is tilting. 
In particular $\ehomo{\KK^{\rm{b}}(\proj{A})}{T}$ is derived equivalent to $A$.  
\end{prop}

If we drop the assumption that $A$ is symmetric, 
an Okuyama-Rickard complex is not necessarily a tilting complex but still a silting complex. 

\begin{definition}
Let $B$ be a finite dimensional $k$-algebra. 
For any $i\in E$, we say that $B$ is the \emph{tilting mutation of $A$ with respect to $i$} 
if $B\simeq \ehomo{\KK^{\rm b}(\proj{A})}{T(E\backslash \{i\})}$ and 
write $A\xrightarrow{i} B$ or $B=\mu_{i}(A)$.
\end{definition}

The aim of this paper is to introduce mutation of Brauer trees 
and study the relationship with tilting mutation of Brauer tree algebras. 

Let us recall the definitions of Brauer trees and Brauer tree algebras 

\begin{definition}\cite{Alp, GR}
A \emph{Brauer graph} $G$ is a finite connected graph, together with the following data: 
\begin{enumerate}[(i)]
\item There exists a cyclic ordering of the edges adjacent to each vertex, 
usually described by the clockwise ordering given by a fixed planar representation of $G$;
\item For each vertex $v$, there exists a positive integer $m_{v}$ assigned to $v$, called the \emph{multiplicity}. 
We call a vertex $v$ \emph{exceptional} if $m_{v}>1$. 
\end{enumerate}
A \emph{Brauer tree} $G$ is a Brauer graph which is a tree and having at most one exceptional vertex. 

A \emph{Brauer tree algebra} $A=A_G$ is a basic algebra given by a Brauer tree $G$ as follows: 
\begin{enumerate}[(i)]
\item There exists a one-to-one correspondence between simple $A$-modules $S_{i}$ and edges $i$ of $G$; 
\item For any edge $i$ of $G$, 
the projective indecomposable $A$-module $P_{i}$ has ${\rm{soc}}(P_{i})\simeq P_{i}/{\rm{rad}}(P_{i})$ 
and ${\rm{rad}}(P_{i})/{\rm{soc}}(P_{i})$ is the direct sum of two uniserial modules 
whose composition factors are, for the cyclic ordering $(i, i_{1}, \cdots, i_{a}, i)$ of the edges adjacent to a vertex $v$, 
$S_{i_{1}}, \cdots, S_{i_{a}}, S_{i}, S_{i_{1}}, \cdots, S_{i_{a}}$ (from the top to the socle) 
where $S_{i}$ appears $m_{v}-1$ times. 
\end{enumerate}
\end{definition}

Note that $A_G$ is uniquely determined by $G$ up to isomorphism. 
Moreover it is a symmetric algebra. 

We say that a Brauer tree $G$ is a \emph{star} if there exists a vertex $v$ of $G$ such that 
all edges of $G$ appear in the cyclic ordering adjacent to $v$. 

The following well-known result shows that derived equivalence classes of Brauer tree algebras 
are determined by certain numerical invariants.

\begin{thm}\label{thm:star} \cite[Theorem 4.2]{Ri}
\begin{enumerate}[{\rm (1)}]
\item For any Brauer tree algebra $A$, 
there exists a tilting complex $P\in\KK^{\rm{b}}(\proj{A})$ such that the endomorphism algebra $\ehomo{\KK^{\rm{b}}(\proj{A})}{P}$ 
of $P$ is a Brauer tree algebra for a star with exceptional vertex in the center if it exists. 
\item In particular, derived equivalence classes of Brauer tree algebras 
are determined by the number of the edges and the multiplicities of the vertices. 
\end{enumerate}
\end{thm}

\section{Mutating Brauer trees}\label{main}

Let us start with the definition of mutation of Brauer trees. 

\begin{definition}\label{mutation of Brauer tree}
Let $G$ be a Brauer tree and $i$ be an edge of $G$. 
We define a Brauer tree $\mu_i(G)$ which is called \emph{mutation} of $G$ with respect to $i$ as follows.
\begin{enumerate}[{\rm (1)}]
\item The edge $i$ is an internal edge:
Let $(i, i_1,\cdots,i_a,i)$ and $(i,j_1,\cdots,j_b,i)$ be the cyclic orderings containing $i$. 
Let $(i_1,g_1,\cdots,g_c,i_1)$ and $(j_1,h_1,\cdots,h_d,j_1)$ be the cyclic orderings containing $i_1$ and $j_1$, 
which do not contain $i$. 
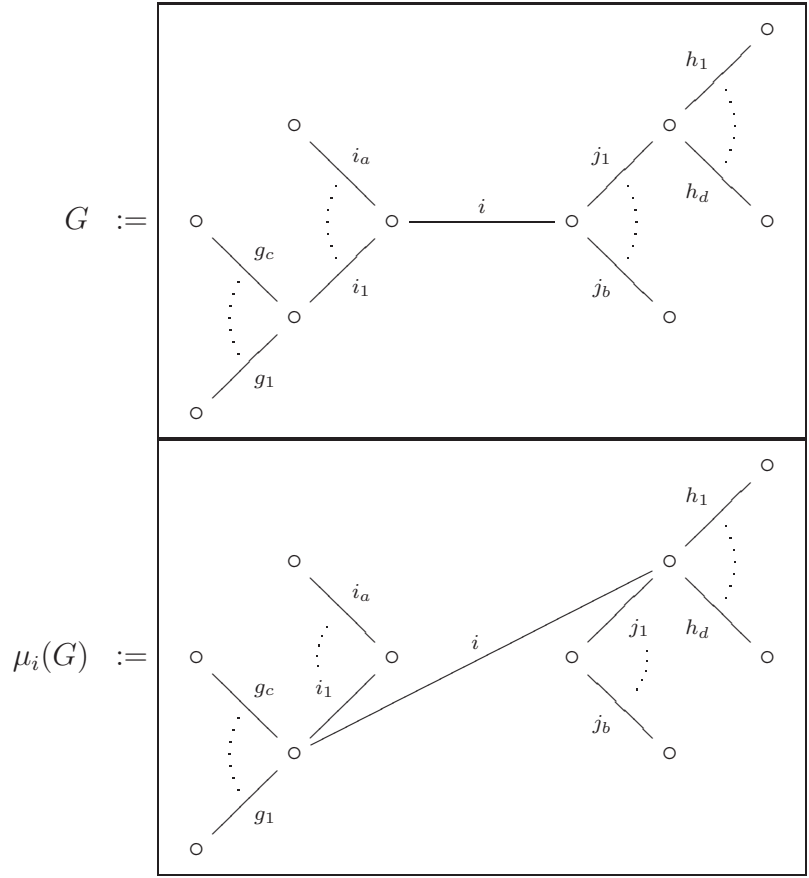
\[
\begin{array}{rl}
G&:=
\framebox{$
\begin{array}{c}
\xymatrix{
&&& &  & & \circ \ar@{-}[ld]^{}="d1"_{h_1}   \\
&\circ \ar@{-}[rd]^{i_a}_{}="a1" &       &    &   & \circ \ar@{-}[ld]_{j_1}^{}="b1" & \\
\circ \ar@{-}[dr]^{g_c}_{}="c1" &  & \circ \ar@{-}[rr]^{i}& & \circ &  & \circ \ar@{-}[ul]^{h_d}_{}="d2" \\
&\circ \ar@{-}[ru]_{i_1}^{}="a2"  &       &&       & \circ \ar@{-}[lu]^{j_b}_{}="b2" & \\
\circ \ar@{-}[ru]^{}="c2"_{g_1} &&&&&&
\ar@{.}@/^/"a2";"a1" 
\ar@{.}@/^/"b1";"b2"
\ar@{.}@/^/"d1";"d2"
\ar@{.}@/^/"c2";"c1"
}
\end{array}$}\\ 
\mu_i(G)&:=
\framebox{$
\begin{array}{c}
\xymatrix{
&&& &  & & \circ \ar@{-}[ld]^{}="d1"_{h_1}   \\
&\circ \ar@{-}[rd]^{i_a}_{}="a1" & && & \circ \ar@{-}[ld]^{j_1}="b1" \ar@{-}[ddllll]_{i} &\\
\circ \ar@{-}[dr]^{g_c}_{}="c1"& & \circ & & \circ && \circ \ar@{-}[ul]^{h_d}_{}="d2" \\
&\circ \ar@{-}[ru]^{i_1}="a2" & && & \circ \ar@{-}[lu]^{j_b}_{}="b2" & \\
\circ \ar@{-}[ru]^{}="c2"_{g_1} &&&&&&
\ar@{.}@/^/"a2";"a1" 
\ar@{.}@/^/"b1";"b2"
\ar@{.}@/^/"d1";"d2"
\ar@{.}@/^/"c2";"c1"
}
\end{array}$}
\end{array}\]
The multiplicities of vertices do not change.

\item The edge $i$ is an external edge: 
Let $(i,j_1,\cdots,j_b,i)$ be the cyclic ordering containing $i$. 
Let $(j_1,h_1,\cdots,h_d,j_1)$ be the cyclic ordering containing $j_1$, 
which does not contain $i$. 
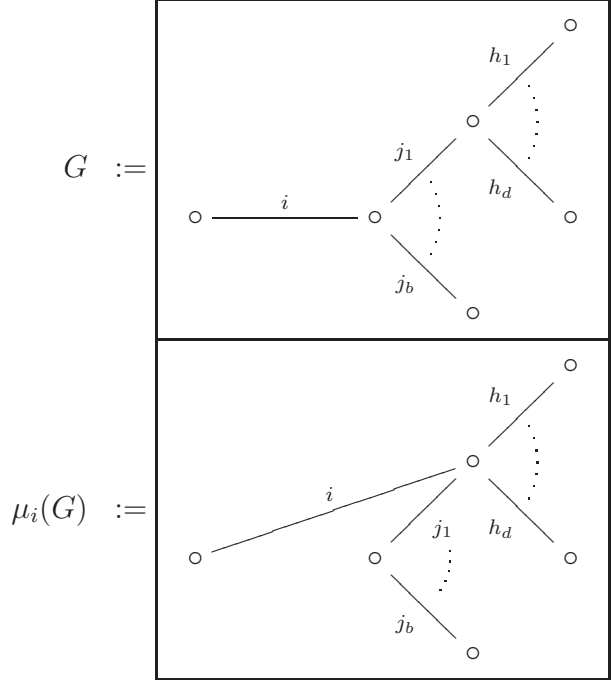
\[
\begin{array}{rl}
G&:=
\framebox{$
\begin{array}{c}
\xymatrix{
&&&& \circ \ar@{-}[ld]^{}="d1"_{h_1}   \\
&    &   & \circ \ar@{-}[ld]_{j_1}^{}="b1" & \\
 \circ \ar@{-}[rr]^{i}& & \circ &  & \circ \ar@{-}[ul]^{h_d}_{}="d2" \\
&&       & \circ \ar@{-}[lu]^{j_b}_{}="b2" & 
\ar@{.}@/^/"b1";"b2"
\ar@{.}@/^/"d1";"d2"
}\end{array}$} \\
\mu_i(G)&:=
\framebox{$
\begin{array}{c}
\xymatrix{
& &  & & \circ \ar@{-}[ld]^{}="d1"_{h_1}   \\
 && & \circ \ar@{-}[ld]^{j_1}="b1" \ar@{-}[dlll]_{i} &\\
 \circ & & \circ && \circ \ar@{-}[ul]^{h_d}_{}="d2" \\
 && & \circ \ar@{-}[lu]^{j_b}_{}="b2" & 
\ar@{.}@/^/"b1";"b2"
\ar@{.}@/^/"d1";"d2"
}\end{array}$}
\end{array}\]
The multiplicities of vertices do not change.
\end{enumerate}
\end{definition}

Now we state the main theorem in this paper. 

\begin{thm}\label{thm:mutation of Brauer tree}
For any Brauer tree $G$ and any edge $i$ of $G$, 
we have an isomorphism $\mu_i(A_G)\simeq A_{\mu_i(G)}$ of $k$-algebras, 
sending each idempotent $e_j$ to $e_j$. 
\end{thm}

We prove the theorem above in the next section. 
This can be regarded as an analogue of derived equivalences associated with Bernstein-Gelfand-Ponomarev reflection of quivers 
\cite{BGP} and Derksen-Weyman-Zelevinsky mutation of quivers with potentials \cite{DWZ, BIRS}. 

\begin{remark}\label{KZ and Z}
Note that a special case of Theorem \ref{thm:mutation of Brauer tree} was given in \cite{KZ} and \cite{Z}, 
where they only considered the case (2) in Definition \ref{mutation of Brauer tree} such that 
there are no exceptional vertices.
\end{remark}

\begin{ex}
We give graphs of mutation of Brauer trees. 
\begin{enumerate}[(1)]
\item Brauer trees with 3 edges and without exceptional vertices: 
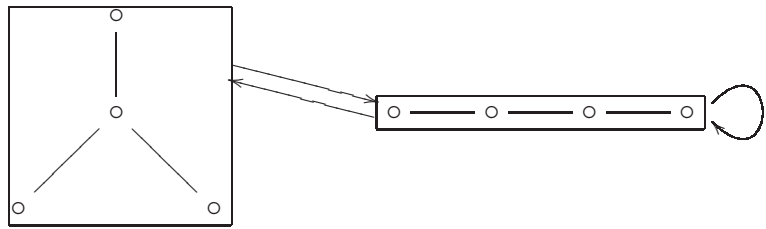
\[\xymatrix{
      & \circ \ar@{-}[d]&       & &       &       &       &\\
      & \circ \ar@{-}[dr]\ar@{-}[dl] &       & & \circ \ar@{-}[r] & \circ \ar@{-}[r] & \circ \ar@{-}[r] & \circ \\
\circ &       & \circ & &       &&&
\save "1,1"."3,3"*[F]\frm{}="a" \restore
\save "2,5"."2,8"*[F]\frm{}="b" \restore
\ar@<0.2em>"a";"b"
\ar@<0.3em>"b";"a"
\ar@(ru, dr) "2,8"+<0.8em,0.3em>;"2,8"+<0.8em, -0.3em>
}\]
\item Brauer trees with 4 edges and without exceptional vertices:
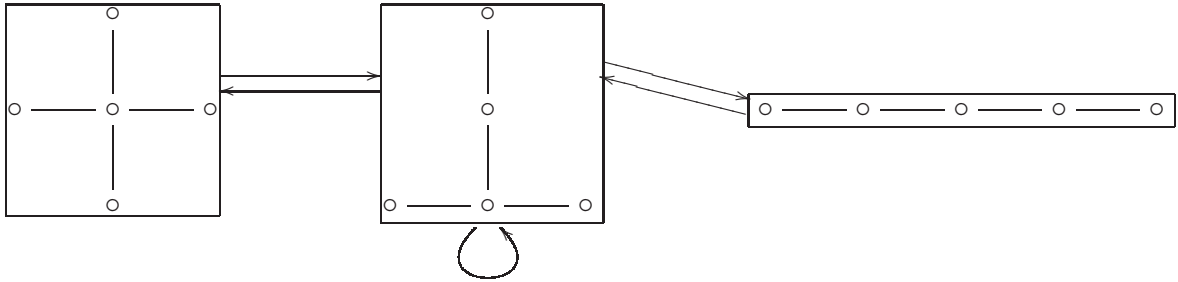
\[\xymatrix{
      & \circ \ar@{-}[d] &       & &       & \circ \ar@{-}[d] &       & &       &       &       &       &       \\
\circ \ar@{-}[r] & \circ \ar@{-}[r] \ar@{-}[d] & \circ & &  & \circ \ar@{-}[d] &       & & 
\circ \ar@{-}[r] & \circ \ar@{-}[r] & \circ \ar@{-}[r] & \circ \ar@{-}[r] & \circ \\
      & \circ &       & & \circ \ar@{-}[r] & \circ \ar@{-}[r]  & \circ & &       &       &       &       & \\
\save "1,1"."3,3"*[F]\frm{}="a"\restore
\save "1,5"."3,7"*[F]\frm{}="b" \restore
\save "2,9"."2,13"*[F]\frm{}="c"\restore 
\ar@<-2em>"a";"b"
\ar@<2.5em>"b";"a"
\ar@<0.2em>"b";"c"
\ar@<0.3em>"c";"b"
\ar@(dl, dr) "3,6"+<-0.4em, -0.7em>;"3,6"+<0.4em, -0.7em>
}\]
\item Brauer trees with 5 edges and without exceptional vertices:
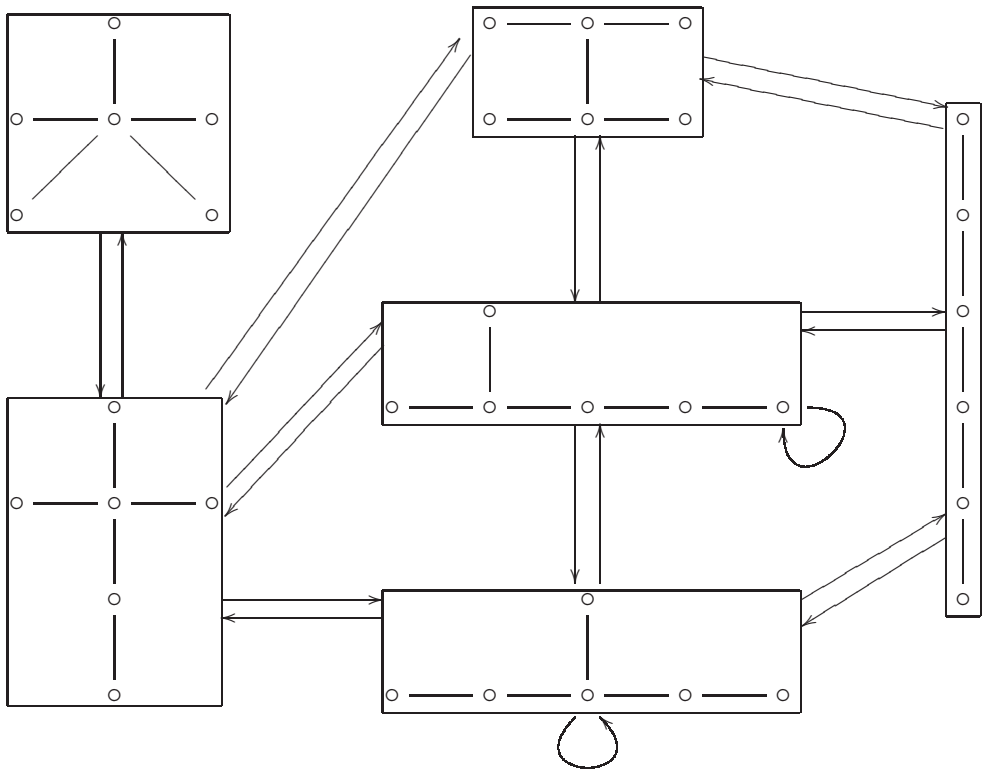
\[\xymatrix{
      & \circ \ar@{-}[d] &       & &       & \circ \ar@{-}[r] & \circ \ar@{-}[d]\ar@{-}[r] & \circ &       & &       \\
\circ & \circ \ar@{-}[l]\ar@{-}[ld]\ar@{-}[r]\ar@{-}[rd] & \circ & & & \circ \ar@{-}[r] & \circ \ar@{-}[r] & \circ & & 
& \circ \ar@{-}[d] \\
\circ &       & \circ & &       &       &       &       &       & & \circ \ar@{-}[d] \\
      &       &       & &       & \circ \ar@{-}[d] &       &       &       & & \circ \ar@{-}[d] \\
      & \circ \ar@{-}[d] &       & & \circ \ar@{-}[r] & \circ \ar@{-}[r] & \circ \ar@{-}[r] & \circ \ar@{-}[r] & \circ & 
& \circ \ar@{-}[d] \\
\circ & \circ \ar@{-}[d]\ar@{-}[l]\ar@{-}[r] & \circ & &       &       &       &       &       & & \circ \ar@{-}[d] \\
      & \circ \ar@{-}[d] &       & &       &       & \circ \ar@{-}[d] &       &       & & \circ \\
      & \circ &       & & \circ \ar@{-}[r] & \circ \ar@{-}[r] & \circ \ar@{-}[r] & \circ \ar@{-}[r] & \circ & &       \\
\save "1,1"."3,3"*[F]\frm{}="a"\restore
\save "5,1"."8,3"*[F]\frm{}="b"\restore
\save "1,6"."2,8"*[F]\frm{}="c"\restore
\save "4,5"."5,9"*[F]\frm{}="d"\restore
\save "7,5"."8,9"*[F]\frm{}="e"\restore
\save "2,11"."7,11"*[F]\frm{}="f"\restore
\ar@<2.7em> "a";"b"
\ar@<-3.4em> "b";"a"
\ar@<0.5em> "5,3";"c"
\ar@<0.3em> "c"+<-10pt, -10pt>;"5,3"
\ar "6,3";"d"
\ar@<0.7em> "d"+<-9pt, -7pt>;"6,3"+<-1pt, 1pt> 
\ar "7,3";"e"
\ar@<0.6em> "e";"7,3"
\ar@<-0.4em> "2,7";"4,7"
\ar@<-0.4em> "4,7";"2,7"
\ar@<-0.4em> "5,7";"7,7"
\ar@<-0.4em> "7,7";"5,7"
\ar@<0.3em> "c";"f"
\ar@<0.4em> "f";"c"
\ar "4,9"+<7pt, 0pt>;"4,11"
\ar@<0.6em> "4,11";"4,9"+<7pt, 0pt>
\ar "7,9"+<7pt, 0pt>;"6,11"
\ar@<0.7em> "6,11"+<-11pt,-6pt>;"7,9"+<3pt, -3pt>
\ar@<0.3em>@(dl, dr) "8,7"+<-0.4em, -1em>;"8,7"+<0.4em, -1em>
\ar@(r, d) "5,9"+<0.8em, 0em>;{"5,9"+<0em, -0.7em>}
}\]
\item Brauer trees with 3 edges and exceptional vertex $\bullet$: 
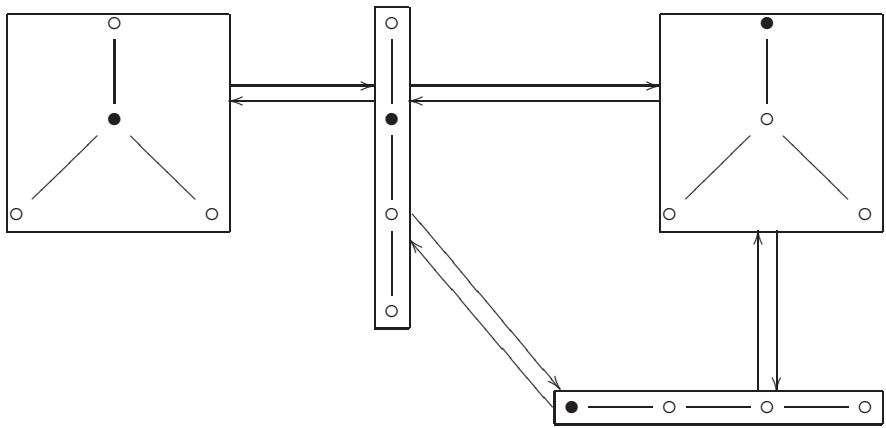
\[\xymatrix{
      & \circ \ar@{-}[d]   &       & & \circ \ar@{-}[d]   & &         &        & \bullet \ar@{-}[d] &       \\ 
      & \bullet \ar@{-}[dl] \ar@{-}[dr] &       & & \bullet \ar@{-}[d] & &         &        & \circ \ar@{-}[dl] \ar@{-}[dr]  &       \\
\circ &         & \circ & & \circ \ar@{-}[d]   & &         & \circ  &         & \circ \\ 
      &         &       & & \circ  & &         &        &         &       \\
      &         &       & &         & & \bullet \ar@{-}[r] & \circ \ar@{-}[r]  & \circ \ar@{-}[r]  & \circ 
\save "1,1"."3,3"*[F]\frm{}="a"\restore
\save "1,5"."4,5"*[F]\frm{}="b"\restore
\save "1,8"."3,10"*[F]\frm{}="c"\restore
\save "5,7"."5,10"*[F]\frm{}="d"\restore
\save "3,5"+<7pt,0pt>="e"\restore
\save "3,5"+<7pt,-0.8em>="f"\restore
\save "d"+<-0.6em, 0em>="g"\restore
\save "3,9"+<0em, -0.5em>="h"\restore
\save "5,9"="i"\restore
\ar@<-2em> "a";"b"
\ar@<2.5em> "b";"a"
\ar@<-2em> "b";"c"
\ar@<2.5em> "c";"b" 
\ar@<0.1em> "e";"d"
\ar "g";"f"
\ar@<0.3em> "h";"i"
\ar@<0.3em> "i";"h"
}\]
\end{enumerate}
\end{ex}

We give some applications of Theorem \ref{thm:mutation of Brauer tree}. 

For edges $i$ and $j$ in a Brauer tree we say that $j$ \emph{follows} $i$ 
if there exists a cyclic ordering of the form $(\cdots, i, j, \cdots)$. 

\begin{thm}\label{tilting formula}
Let $A$ be a Brauer tree algebra. 
Then for any edges $i, j$ of the Brauer tree of $A$, we have the following relations:
\begin{enumerate}[$(1)$]
\item $(\mu_{i})^{s}(A)\simeq A$ for some positive integer $s$;
\item $\mu_{j}\mu_{i}(A)\simeq \mu_{i}\mu_{j}(A)$ if $i$ does not follow $j$ and $j$ does not follow $i$; 
\item $\mu_{i}\mu_{j}\mu_{i}(A)\simeq \mu_{i}\mu_{j}(A)$ if $i$ does not follow $j$ and $j$ follows $i$;
\item $\mu_i\mu_j\mu_i(A)\simeq \mu_j\mu_i\mu_j(A)$ if $i$ follows $j$ and $j$ follows $i$.
\end{enumerate}
\end{thm}

\begin{proof}
Let $G$ be the Brauer tree of $A$. 
By Theorem \ref{thm:mutation of Brauer tree}, 
we only have to show the corresponding isomorphisms of Brauer trees for $G$. 
\begin{enumerate}[$(1)$]
\item We denote by ${\rm{Br}}(n, m)$ the set of labeled Brauer trees 
which have $n$ edges and multiplicity $m$ of the exceptional vertex. 
We can regard mutation as group action to ${\rm{Br}}(n, m)$. 
Since ${\rm{Br}}(n, m)$ is a finite set, 
the order of mutation is finite. 
\item This assertion can be checked easily. 
\item Let $j'$ follow $i$ being not $j$ and $j_{1}, j_{2}$ follow $j$. 
We have the following mutation: 
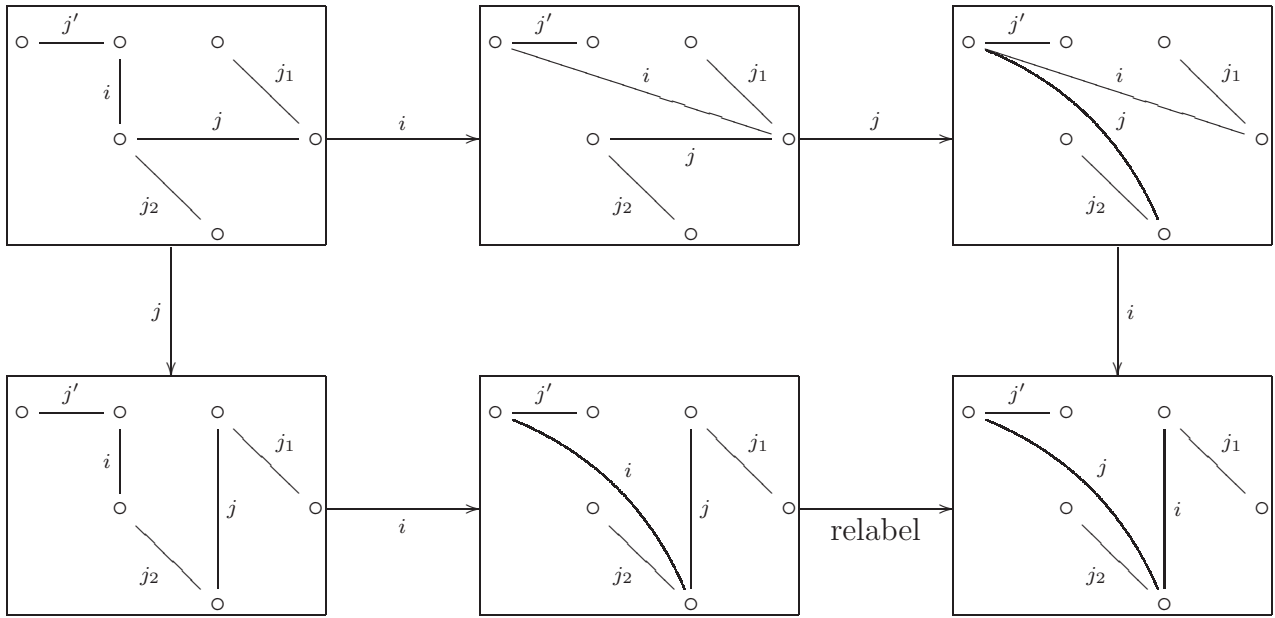
\[\xymatrix{
\circ \ar@{-}[r]^{j'} & \circ \ar@{-}[d]_{i} & \circ \ar@{-}[dr]^{j_1} &       
& & \circ \ar@{-}[r]^{j'} \ar@{-}[drrr]^{i} & \circ & \circ \ar@{-}[dr]^{j_1} &       
& & \circ \ar@{-}[r]^{j'} \ar@{-}[drrr]^{i} \ar@{-}@/^1pc/[ddrr]^(0.6){j} & \circ & \circ \ar@{-}[dr]^{j_1} &       \\
      & \circ \ar@{-}[rr]^{j} \ar@{-}[dr]_{j_2} &       & \circ 
& &       & \circ \ar@{-}[rr]_{j} \ar@{-}[rd]_{j_2} &       & \circ 
& &       & \circ \ar@{-}[dr]_{j_2} &       & \circ \\
      &       & \circ &       & &       &       & \circ &       & &       &       & \circ &       \\
      &       &       &       & &       &       &       &       & &       &       &       &       \\
\circ \ar@{-}[r]^{j'} & \circ \ar@{-}[d]_{i} & \circ \ar@{-}[rd]^{j_1} \ar@{-}[dd]^{j} &       
& & \circ \ar@{-}[r]^{j'} \ar@{-}@/^1pc/[rrdd]^{i} & \circ & \circ \ar@{-}[dd]^{j} \ar@{-}[dr]^{j_1} &       
& & \circ \ar@{-}[r]^{j'} \ar@{-}@/^1pc/[ddrr]^{j} & \circ & \circ \ar@{-}[rd]^{j_1} \ar@{-}[dd]^{i} &       \\
      & \circ \ar@{-}[dr]_{j_2} &       & \circ 
& &       & \circ \ar@{-}[dr]_{j_2} &       & \circ 
& &       & \circ \ar@{-}[dr]_{j_2} &       & \circ \\
      &       & \circ &       & &       &       & \circ &       & &       &       & \circ &   
\save "1,1"+<-0.5em, 1.2em>."3,4"*[F]\frm{}\restore
\save "1,6"+<-0.5em, 1.2em>."3,9"*[F]\frm{}\restore 
\save "1,11"+<-0.5em, 1.2em>."3,14"*[F]\frm{}\restore
\save "5,1"+<-0.5em, 1.2em>."7,4"*[F]\frm{}\restore
\save "5,6"+<-0.5em, 1.2em>."7,9"*[F]\frm{}\restore
\save "5,11"+<-0.5em, 1.2em>."7,14"*[F]\frm{}\restore
\ar "2,4"+<0.3em,0em>;{"2,6"+<-0.5em,0em>}^{i}
\ar "2,9"+<0.3em,0em>;{"2,11"+<-0.5em,0em>}^{j}
\ar@<-1.5em> "3,13"+<0em,-0.4em>;{"5,13"+<0em,1.2em>}^{i}
\ar@<-1.5em> "3,3"+<0em,-0.4em>;{"5,3"+<0pt, 1.2em>}_{j}
\ar "6,4"+<0.3em,0em>;{"6,6"+<-0.5em,0em>}_{i}
\ar "6,9"+<0.3em,0em>;{"6,11"+<-0.5em,0em>}_{\mbox{relabel}}
}\]
Hence we obtain $\mu_{i}\mu_{j}\mu_{i}(G)\simeq \mu_{i}\mu_{j}(G)$. 
\item Let $i'$ follow $i$ being not $j$ and $j'$ follow $j$ being not $i$. 
We have the following mutation:
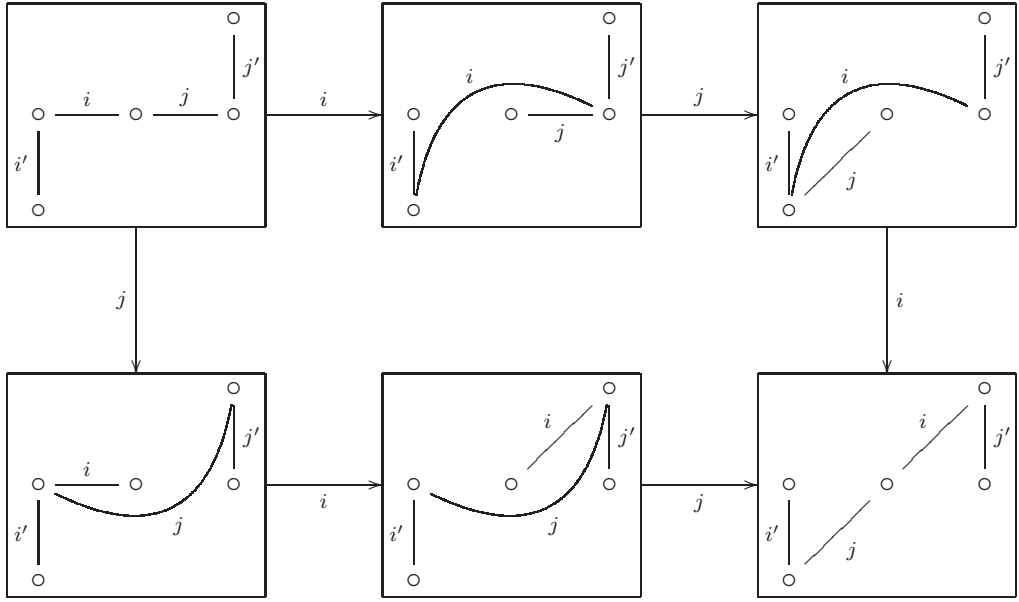
\[\xymatrix{
      &       & \circ \ar@{-}[d]^{j'} 
& &       &       & \circ \ar@{-}[d]^{j'} 
& &       &       & \circ \ar@{-}[d]^{j'} \\
\circ \ar@{-}[r]^{i} \ar@{-}[d]_{i'} & \circ \ar@{-}[r]^{j} & \circ 
& & \circ \ar@{-}[d]_{i'} & \circ \ar@{-}[r]_{j} & \circ      
& & \circ \ar@{-}[d]_{i'} & \circ \ar@{-}[dl]^{j} & \circ \\
\circ &       &       
& & \circ \ar@{-}@/^2.5pc/[rru]^{i} &       &       
& & \circ \ar@{-}@/^2.5pc/[rru]^{i} &       &       \\
&&&&&&&&&&\\
      &       & \circ \ar@{-}[d]^{j'} 
& &       &       & \circ \ar@{-}[d]^{j'}
& &       &       & \circ \ar@{-}[d]^{j'} \\
\circ \ar@{-}[r]^{i} \ar@{-}[d]_{i'} \ar@{-}@/_2.5pc/[rru]_{j} & \circ & \circ 
& & \circ \ar@{-}[d]_{i'} \ar@{-}@/_2.5pc/[rru]_{j} & \circ \ar@{-}[ru]^{i} & \circ  
& & \circ \ar@{-}[d]_{i'} & \circ \ar@{-}[ru]^{i} \ar@{-}[ld]^{j} & \circ \\
\circ &       &       & & \circ &       &       & & \circ &       &      
\save "1,1"+<-1em, 0.5em>.{"3,3"+<1em, -0.5em>}*[F]\frm{}\restore 
\save "1,5"+<-1em, 0.5em>.{"3,7"+<1em, -0.5em>}*[F]\frm{}\restore 
\save "1,9"+<-1em, 0.5em>.{"3,11"+<1em, -0.5em>}*[F]\frm{}\restore 
\save "5,1"+<-1em, 0.5em>.{"7,3"+<1em, -0.5em>}*[F]\frm{}\restore 
\save "5,5"+<-1em, 0.5em>.{"7,7"+<1em, -0.5em>}*[F]\frm{}\restore 
\save "5,9"+<-1em, 0.5em>.{"7,11"+<1em, -0.5em>}*[F]\frm{}\restore 
\ar "2,3"+<1em, 0em>;"2,5"+<-1em, 0em>^{i}
\ar "2,7"+<1em, 0em>;"2,9"+<-1em, 0em>^{j}
\ar "3,10"+<0em, -0.5em>;"5,10"+<0em, 0.5em>^{i}
\ar "3,2"+<0em, -0.5em>;"5,2"+<0em, 0.5em>_{j}
\ar "6,3"+<1em, 0em>;"6,5"+<-1em, 0em>_{i}
\ar "6,7"+<1em, 0em>;"6,9"+<-1em, 0em>_{j}
}\]
Hence one gets $\mu_i\mu_j\mu_i(G)\simeq \mu_j\mu_i\mu_j(G)$.
\end{enumerate}
\end{proof}

As another application of Theorem \ref{thm:mutation of Brauer tree}, 
we give a simple proof of the following stronger statement than Theorem \ref{thm:star}. 

\begin{cor}\label{cor:star length 2}
Let $G$ be a Brauer tree and $\ell\geq1$. 
For every vertex $v$ of $G$ there exists a tilting complex $T\in\KK^{\rm b}(\proj{A_G})$ of the form 
$(\cdots\to 0\to P^0\to P^1\to 0\to\cdots)$ with $P^0, P^1\in\proj{A_G}$ 
such that the Brauer tree of $\ehomo{\KK^{\rm b}(\proj{A})}{T}$ is a star with the vertex $v$ in the center.
\end{cor}

To prove the corollary above, we need the following preliminary results. 

\begin{lemma}\label{mutating to star}
For any Brauer tree $G$ and any vertex $v$ of $G$,
there exists a sequence $i_1,...,i_\ell$ of distinct edges
such that $\mu_{i_\ell}\cdots\mu_{i_1}(G)$ is a star with the
vertex $v$ in the center.
\end{lemma}
\begin{proof}
Take an edge $i_1$ which is followed by an edge $j$ with the vertex $v$: 
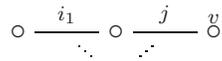
\[\xymatrix{
\circ \ar@{-}[r]^{i_1}_{}="a" & \circ \ar@{-}[r]_{}="b"^{j} \ar@{}[r]^(1){v} & \circ  
\ar@{.}@/_1pc/"a";"b"
}\]
By Theorem \ref{thm:mutation of Brauer tree}, we see that the Brauer tree $\mu_{i_1}(G)$ is of the form 
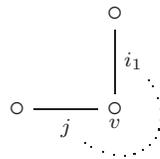
\[\xymatrix{
      & \circ \ar@{-}[d]^{i_1}="a" & \\
\circ \ar@{-}[r]_{j}="b" \ar@{}[r]_(1){v} & \circ 
\ar@{.}@(rd, rd)"a";"b"
}\]
which says that the edge $i_1$ has the vertex $v$. 
Continuing the argument, we obtain a star with the vertex $v$ in the center. 
\end{proof}

\begin{lemma}\label{distinct edges} 
Let $G$ be a Brauer tree and $\ell\geq1$. 
If $i_1,\cdots,i_\ell$ are distinct edges of $G$, 
then there exists a tilting complex $T\in\KK^{\rm b}(\proj{A_G})$ of the form $(\cdots\to 0\to P^0\to P^1\to 0\to\cdots)$ 
with $P^0, P^1\in\proj{A_G}$
such that $\mu_{i_\ell}\cdots\mu_{i_1}(A_G)\simeq\ehomo{\KK^{\rm b}(\proj{A_G})}{T}$. 
\end{lemma}
\begin{proof}
For any edge $i$ of $G$ 
we have a derived equivalence $F_i:\KK^{\rm b}(\proj{\mu_i(A_G)})\xrightarrow{\sim}\KK^{\rm b}(\proj{A_G})$ 
given by the Okuyama-Rickard complex $T(E\backslash \{i\})$ of $A_G$, 
Put $T_\ell:=F_{i_1}\cdots F_{i_\ell}(\mu_{i_\ell}\cdots\mu_{i_1}(A_G))$, 
which is a tilting complex in $\KK^{\rm b}(\proj{A_G})$.
We show that $T_\ell$ has a form
\begin{equation}\label{our complex}
T_\ell=\begin{cases}
\begin{array}{c}
\xymatrix@!R=0.5pt{
 (0{\rm{th}}) & (1{\rm{st}}) & \\
P \ar[r] & 0 \\
\oplus & \\
Q^0 \ar[r] & Q^1
}
\end{array}
\end{cases}
\end{equation}
where $P=\bigoplus_{j\in E\backslash \{i_1,\cdots,i_\ell\}}P_j$ and $Q^0, Q^1\in\proj{A_G}$.
We use induction on $\ell\geq1$. 
If $\ell=1$, then one observes $T_1=T(E\backslash \{i_1\})$. 
This says that $T_1$ is of the form (\ref{our complex}). 
Assume $\ell\geq2$. 
It follows from the induction hypothesis that $P_{i_\ell}$ is a direct summand of $T_{\ell-1}$ and 
the complement $T_{\ell-1}\backslash P_{i_\ell}$ is of the form $R:=(\cdots\to 0\to R^0\to R^1\to 0\to\cdots)$. 
We see that $T_\ell$ is given by the direct sum of $R$ and a complex $P'$ 
which admits a triangle $P'\to R\to P_{i_\ell}\to P'[1]$: see \cite{AI}. 
Since $P_{i_\ell}$ and $R$ concentrate on degree 0 and $(0,1)$ respectively, 
one obtains that $P'$ is of the form $(\cdots\to0\to (P')^0\to (P')^1\to0\to\cdots)$. 
This implies that $T_\ell$ is of the form (\ref{our complex}). 
\end{proof}

Now Corollary \ref{cor:star length 2} is an immediate consequence of Lemma \ref{mutating to star} and Lemma \ref{distinct edges}. 
\qed
\medskip
 
The following result is a direct consequence of Corollary \ref{cor:star length 2}. 

\begin{cor}
Let $A$ be a Brauer tree algebra. 
Any basic algebra which is derived equivalent to $A$ is obtained from $A$ by iterated tilting mutation. 
\end{cor}

More generally, the statement above is shown for representation-finite symmetric algebras in \cite{A}.



\section{Proof of main theorem}\label{proof}

In this section we prove the main theorem of this paper. 

We define an $(n\times n)$-matrix $C^{A}$ as $C^{A}_{ij}={\rm{dim}}_{k}\homo{A}{P_{i}}{P_{j}}$ for any $i,j\in E$, 
called the \emph{Cartan matrix} of $A$. 
Note that if $A$ is a symmetric algebra, then we have $C^{A}_{ij}=C^{A}_{ji}$ for any $i, j\in E$. 

We have the following property. 

\begin{lemma}\label{lemma:Brauer tree and Cartan matrix}
Let $G$ be a Brauer tree. 
Then the Cartan matrix $C^{A_G}$ of $A_G$ is determined as follows: 
\[C^{A_G}_{ij}=
\begin{cases}
\ m_v+m_u & \mbox{if } i=j\ \mbox{and } \mbox{the ends of }i\ \mbox{are } u, v; \\
\ m_v & \mbox{if } i\not=j\ \mbox{and } i, j\ \mbox{have a common end } v; \\
\ 0 & \mbox{otherwise}.
\end{cases}\]
\end{lemma}

Conversely, the Cartan matrix $C^A$ of a Brauer tree algebra $A$ and the data of extensions among simple $A$-modules 
determine the Brauer tree of $A$ by the following: 

\begin{method}\label{method}
Let $A$ be a Brauer tree algebra. 
Assume that the Cartan matrix $C^{A}$ of $A$ and $\dim\ext{A}{1}{S_i}{S_j}$ for any $i,j\in E$ are given. 
We explicitly determine the Brauer tree $G$ of $A$. 

(1) We give the cyclic ordering containing each edge. 
Fix any $i\in E$. 
We define a subset $I$ of $E$ by $I=\{j\in E\ |\ C^{A}_{ij}\not=0\}$. 
Since $G$ is a Brauer tree, one has a disjoint union $I=\{i\}\cup I_{0}\cup I_{1}$ satisfying   
$C^{A}_{i_{0}i_{1}}=0$ for any $i_{0}\in I_{0}$ and $i_{1}\in I_{1}$.
Moreover, for any $\ell\in \{0,1\}$ and any $j\in \{i\}\cup I_\ell$ 
there exists uniquely $j'\in \{i\}\cup I_\ell$ such that $\ext{A}{1}{S_j}{S_{j'}}\not=0$. 
Thus we can take sequences 
\[i=i^{0}, i^{1}, \cdots, i^{a}, i^{a+1}=i\ \mbox{in } \{i\}\cup I_{0}\]
\[i=j^{0}, j^{1}, \cdots, j^{b}, j^{b+1}=i\ \mbox{in } \{i\}\cup I_{1}\] 
such that $\ext{A}{1}{S_{i^{x}}}{S_{i^{x+1}}}\not=0$ for any $0\leq x\leq a$ and 
$\ext{A}{1}{S_{j^{y}}}{S_{j^{y+1}}}\not=0$ for any $0\leq y\leq b$. 
Hence we can explicitly determine the cyclic ordering containing $i$:
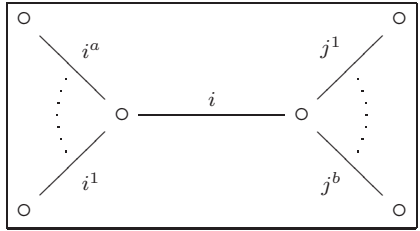
\[\xymatrix{
\circ \ar@{-}[rd]^{i^{a}}_{}="a1" &       &    &   & \circ \ar@{-}[ld]_{j^1}^{}="b1" \\
                  & \circ \ar@{-}[rr]^{i}& & \circ & \\
\circ \ar@{-}[ru]_{i^1}^{}="a2"  &       &&       & \circ \ar@{-}[lu]^{j^{b}}_{}="b2"
\ar@{.}@/^/"a2";"a1" 
\ar@{.}@/^/"b1";"b2"
\save "1,1"."3,5"*[F]\frm{}\restore
}\]

(2) We give the position of the exceptional vertex if it exists. 
Note that the exceptional vertex exists if and only if there is $i\in E$ satisfying $C^A_{ii}>2$. 
Put $\mathcal{E}:=\{i\in E\ |\ C^A_{ii}>2\}$ and assume that $\mathcal{E}$ is not an empty set. 
Since the Brauer tree $G$ has only one exceptional vertex, 
we observe that all edges in $\mathcal{E}$ have a common vertex $v$ and 
any edge having the vertex $v$ belongs to $\mathcal{E}$. 
Thus the vertex $v$ is exceptional. 
\qed
\end{method}

We show the following easy observation. 

\begin{prop}\label{prop:closed under derived}
Let $A$ and $B$ be derived equivalent symmetric $k$-algebras. 
If $A$ is a Brauer tree algebra, then so is $B$. 
\end{prop}
\begin{proof}
By Theorem \ref{thm:star}, $A$ is derived equivalent to 
a Brauer tree algebra $C$ for a star with the exceptional vertex in the center if it exists. 
Since $A$ and $B$ are derived equivalent, 
it follows that $B$ and $C$ are also derived equivalent. 
This implies that $B$ is stable equivalent to $C$. 
Note that $C$ is a symmetric Nakayama algebra. 
Hence the assertion follows from \cite[X, Theorem 3.14]{ARS}. 
\end{proof}

We also need the following result. 

\begin{lemma}\label{lemma:ok_stable} \cite[Lemma 2.1]{Ok} 
Let $E_{0}$ be a subset of $E$ and put $e:=\sum_{i\in E_{0}}e_{i}$.
Let $T:=T(E_0)$ be the Okuyama-Rickard complex with respect to $E_0$. 
Assume that $A$ is a symmetric algebra. 
Now the endomorphism algebra $B=\ehomo{\KK^{\rm{b}}(\proj{A})}{T}$ of $T$ is stable equivalent to $A$ and  
we denote the stable equivalence by $F:{\underline{\rm{mod}}}\text{-}A\stackrel{\sim}{\rightarrow}{\underline{\rm{mod}}}\text{-}B$. 
Then the following hold:  
\begin{enumerate}[$(1)$]
\item If $i \not\in E_{0}$, then $F(\syz{}{S_{i}})$ is a simple $B$-module;
\item If $i \in E_{0}$, then $F(Y_i)$ is a simple $B$-module, 
where $Y_i$ is maximal amongst submodules of $P_i$ such that 
any $S_j$ ($j\in E_0$) is not a composition factor of $Y_i/S_i$.
\end{enumerate}
\end{lemma}

Now we are ready to prove Theorem \ref{thm:mutation of Brauer tree}. 

\begin{mpf}
Let $G$ be the Brauer tree of $A$ 
and we use the notation of Definition \ref{mutation of Brauer tree}. 
Since the Okuyama-Rickard complex $T:=T(E\backslash\{i\})$ is tilting by Proposition \ref{prop:ok}, 
$B:=\mu_i(A)$ is a Brauer tree algebra by Proposition \ref{prop:closed under derived}. 
Our goal is to show that the Brauer tree of $B$ coincides with $\mu_i(G)$. 
To do this, we only have to calculate the Cartan matrix of $B$ and extensions among simple $B$-modules.
 
Recall that $T$ is defined as the direct sum of the following complexes:
\[\begin{array}{c}
\xymatrix@!R=0.5pt{
& (0{\rm{th}}) & (1{\rm{st}}) & \\
T_{i}: & P_{i_1}\oplus P_{j_1} \ar[r] & P_{i} &  \\
T_{\ell}: & P_{\ell} \ar[r] & 0 & (\ell\not=i)
}
\end{array}
\]
(If the edge $i$ is external, then replace the above first complex with $P_{i_1}\to P_{i}$ or $P_{j_1}\to P_{i}$.) 

(1) Let $C^A, C^B$ be Cartan matrices of $A, B$. 
We calculate $C^B_{\ell m}$. 
For each $\ell\in E$, we denote by $P^B_{\ell}$ a projective indecomposable $B$-module corresponding to $T_\ell$. 

\begin{enumerate}[(i)]
\item We can easily check $C^B_{\ell m}=C^A_{\ell m}$ for any $\ell\not=i$ and $m\not=i$. 
\item We calculate $C^B_{i\ell}$ for $\ell\not=i$. 
If $\ell\not=i$, then we have equalities 
\[\def\arraystretch{1.5}
\begin{array}{rl}
C^{B}_{i\ell} &= \dim{\homo{B}{P^{B}_{i}}{P^{B}_{\ell}}} \\
           &= \dim{\homo{\KK^{\rm{b}}(\proj{A})}{T_{i}}{T_{\ell}}} \\
           &= \dim{\homo{A}{P_{i_1}}{P_{\ell}}}+\dim{\homo{A}{P_{j_1}}{P_{\ell}}}-\dim{\homo{A}{P_{i}}{P_{\ell}}} \\
           &= C^{A}_{i_1\ell}+C^{A}_{j_1\ell}-C^{A}_{i\ell}.
\end{array}\]
Therefore one sees the following: 
\[C^B_{i\ell}=\begin{cases}
\ 0 & (\ell\in \{i_2,\cdots,i_a\}\ \mbox{or } \{j_2,\cdots,j_b\}); \\
\ C^A_{i_1\ell}\not=0 & (\ell\in\{g_1,\cdots,g_c\}); \\
\ C^A_{j_1\ell}\not=0 & (\ell\in\{h_1,\cdots,h_d\}); \\
\ m_{v(\ell)} & (\ell=i_1\ \mbox{or } j_1); \\
\ 0 & (\mbox{otherwise})
\end{cases}\]
where $v(\ell)$ is the vertex of $\ell$ that $i$ does not have. 
\item We show $C^B_{ii}=m_{v(i_1)}+m_{v(j_1)}$.     
One sees equalities 
\[\def\arraystretch{1.5}
\begin{array}{rl}
C^B_{ii} &= \dim\homo{B}{P^B_i}{P^B_i} \\
         &= \dim\homo{\KK^{\rm b}(\proj{A})}{T_i}{T_i} \\
         &= C^A_{ii}+C^A_{i_1i_1}+C^A_{j_1j_1}+2C^A_{i_1j_1}-2(C^A_{ii_1}+C^A_{ij_1}) \\
         &= m_{v(i_1)}+m_{v(j_1)}. 
\end{array}\]
\end{enumerate}

(2) For each $\ell\in E$, we put $S^{B}_{\ell}=P^{B}_{\ell}/{\rm{rad}}P^{B}_{\ell}$. 
We calculate $\ext{B}{1}{S^{B}_{\ell}}{S^{B}_{m}}$. 
We denote by $F:\sta{A}\to \sta{B}$ 
the stable equivalence between $A$ and $B$ given by $T$. 
By Lemma \ref{lemma:ok_stable}, 
one sees that $F$ sends 
\[X_\ell:=\begin{cases}
\ \syz{}{S_i} & (\ell=i) \\
\ Y_{\ell} & (\ell=i_1\ \mbox{or } j_1) \\
\ S_\ell & (\mbox{otherwise})
\end{cases}\]
to $S^B_\ell$, 
where $Y_\ell$ is a unique submodule of $P_\ell$ whose Loewy series is
$\begin{pmatrix}
S_i \\
S_\ell
\end{pmatrix}$.
\begin{enumerate}[(a)]
\item We can easily check $\ext{B}{1}{S^{B}_{\ell}}{S^{B}_{m}}\simeq \ext{A}{1}{S_\ell}{S_m}$ 
for $\ell, m\not\in \{i, i_1, j_1\}$ or $\ell=m=i$. 
\item We calculate $\ext{B}{1}{S^B_\ell}{S^B_m}$ for $\ell\in\{i,i_1,j_1\}$ and $m\not\in\{i,i_1,j_1\}$. 
One has isomorphisms
\[\def\arraystretch{1.5}
\begin{array}{rl}
\ext{B}{1}{S^B_\ell}{S^B_m} &\simeq \ext{A}{1}{X_\ell}{S_m} \\
                            &\simeq \shomo{A}{\syz{}{X_\ell}}{S_m} \\
                            &\simeq \homo{A}{\syz{}{X_\ell}}{S_m} \\
                            &\simeq \begin{cases}
                             \ \homo{A}{\syz{2}{S_i}}{S_m} & (\ell=i) \\
                             \ \homo{A}{\syz{}{Y_\ell}}{S_m} & (\ell=i_1\ \mbox{or } j_1) \\
                            \end{cases} \\
                            & \begin{cases}
                             \ \not=0 & (\ell=i\ \mbox{and } m\in\{g_1,h_1\}) \\
                             \ \not=0 & ((\ell, m)=(i_1, i_2)\ \mbox{or } (j_1, j_2)) \\
                             \ =0      & (\mbox{otherwise}).
                            \end{cases}
\end{array}\]
Similarly, for $\ell\not\in\{i,i_1,j_1\}$ and $m\in\{i,i_1,j_1\}$ we obtain isomorphisms
\[\def\arraystretch{1.5}
\begin{array}{rl} 
\ext{B}{1}{S^B_\ell}{S^B_m} &\simeq \homo{A}{S_\ell}{\syz{-1}{X_m}} \\
                            &\simeq \begin{cases}
                             \ \homo{A}{S_\ell}{S_i} & (m=i) \\
                             \ \homo{A}{S_\ell}{\syz{-1}{Y_m}} & (m=i_1\ \mbox{or } j_1) 
                            \end{cases} \\
                            & \begin{cases} 
                             \ \not=0 & ((\ell,m)=(i_a,i_1), (g_c, i_1), (j_b,j_1)\ \mbox{or } (h_d,j_1)) \\
                             \ =0 & (\mbox{otherwise}).
                            \end{cases}
\end{array}\]
\item We show $\ext{B}{1}{S^B_\ell}{S^B_i}\not=0$ for $\ell\in\{i_1,j_1\}$. 
One sees isomorphisms
\[\def\arraystretch{1.5}
\begin{array}{rl} 
\ext{B}{1}{S^B_\ell}{S^B_i} &\simeq \ext{A}{1}{Y_\ell}{\syz{}{S_i}} \\
                            &\simeq \shomo{A}{Y_\ell}{S_i} \\
                            &\simeq \homo{A}{Y_\ell}{S_i} \\
                            &\not=0.
\end{array}\]
Similarly, for $\ell\in\{i_1, j_1\}$ we have an isomorphism 
\[\def\arraystretch{1.5}
\begin{array}{rl} 
\ext{B}{1}{S^B_i}{S^B_\ell} &\simeq \homo{A}{S_i}{\syz{-2}{Y_\ell}} \\
                            & \begin{cases}
                             \ =0 & \mbox{if $\ell$ is an internal edge}  \\ 
                             \ \not=0 & \mbox{otherwise} \\
                            \end{cases}
\end{array}\].
\item We calculate $\ext{B}{1}{S^B_\ell}{S^B_m}$ for $\ell, m\in\{i_1,j_1\}$.
If $\ell\not=m$, it follows from (i) that $C^B_{\ell m}=0$, 
which implies $\ext{B}{1}{S^B_\ell}{S^B_m}=0$. 
Let $\ell=m$. 
Since $\homo{A}{Y_\ell}{Y_\ell}\simeq \homo{A}{P_i}{Y_\ell}$, 
we have isomorphisms 
\[\def\arraystretch{1.5}
\begin{array}{rl} 
\ext{B}{1}{S^B_\ell}{S^B_\ell} &\simeq \ext{A}{1}{Y_\ell}{Y_\ell} \\
                               &\simeq \homo{A}{\syz{}{Y_\ell}}{Y_\ell} \\
                               & \begin{cases}
                                \ \not=0 & \mbox{if } \ext{A}{1}{S_{\ell}}{S_i}\not=0\ \mbox{and } m_v>1; \\
                                \ =0 & \mbox{otherwise}
                               \end{cases}
                               
\end{array}\]
where $v$ is the common vertex of $\ell$ and $i$. 
\end{enumerate}

Applying Method \ref{method}, we conclude that the Brauer tree of $B$ is given by $\mu_i(G)$. 
\qed
\end{mpf}


\begin{ac}
The author would like to give his deep gratitude to Osamu Iyama, Shigeo Koshitani, Sefi Ladkani and Joseph Grant, 
who read the paper carefully and gave an important information on earlier literature and a lot of helpful comments and suggestions.
\end{ac}

\end{document}